\documentclass[ECP]{ejpecp} % replace ECP by EJP if needed.  

\SHORTTITLE{A rank-based mean field game} 

\TITLE{A rank-based mean field game in the strong formulation\thanks{This research is supported in part by the National Science Foundation under grant DMS-1613170.}} % \thanks is optional. Insert line breaks with \\

\AUTHORS{%
  Erhan~Bayraktar\footnote{University of Michigan, United States of America.
    \EMAIL{erhan@umich.edu}}
  \and %% remove this line and below if single author
  Yuchong~Zhang\footnote{Columbia University,
    United States of America. \EMAIL{yz2915@columbia.edu} }}%AUTHORS
\KEYWORDS{Mean field games; competition; common noise; rank-dependent interaction; non-local interaction; strong formulation} % Separate items with ;

\AMSSUBJ{60H; 91A} % Edit. Separate items with ;
\SUBMITTED{March 26, 2016} % Edit.
\ACCEPTED{October 6, 2016} % Edit.

\VOLUME{0}
\YEAR{2012}
\PAPERNUM{0}
\DOI{vVOL-PID}

\ABSTRACT{We discuss a natural game of competition and solve the corresponding mean field game with \emph{common noise} when agents' rewards are \emph{rank-dependent}. We use this solution to provide an approximate Nash equilibrium for the finite player game and obtain the rate of convergence.}

\newcommand{\snorm}[1]{\| #1\|_\infty}
\newcommand\numberthis{\addtocounter{equation}{1}\tag{\theequation}}

\begin{document}

%%%%%%%%%%%%%%%%%%%%%%%%%%%%%%%%%%%%%%%%%%%%%%%%%%%%%%%%%%%%%%%%%%%
%%                                                               %%
%% No need for \maketitle.                                       %%
%%                                                               %%
%%%%%%%%%%%%%%%%%%%%%%%%%%%%%%%%%%%%%%%%%%%%%%%%%%%%%%%%%%%%%%%%%%%

%%%%%%%%%%%%%%%%%%%%%%%%%%%%%%%%%%%%%%%%%%%%%%%%%%%%%%%%%%%%%%%%%%%
%%                                                               %%
%% Please replace what follows by the body of your article       %%
%% (up to the bibliography):                                     %%
%%                                                               %%
%%%%%%%%%%%%%%%%%%%%%%%%%%%%%%%%%%%%%%%%%%%%%%%%%%%%%%%%%%%%%%%%%%%

\section{Introduction}
Mean field games (MFGs), introduced independently by \cite{MR2295621} and \cite{MR2346927}, provide a useful approximation for the finite player Nash equilibrium problems in which the players are coupled through their empirical distribution. In particular, the mean field game limit gives an approximate  Nash equilibrium, in which the agents' decision making is decoupled.
In this paper we will consider a particular game in which the interaction of the players is through their ranks. Our main goal is to construct an approximate Nash equilibrium for a finite player game when the agents' dynamics are modulated by common noise.

Rank-based mean field games, which have non-local mean field interactions, have been suggested in \cite{MR2762362} and analyzed more generally by the recent paper by Carmona and Lacker \cite{LackerAAP} using the weak formulation, when there is no common noise. There are currently no results on the rank-dependent mean field games with common noise.
In order to solve the problem with common noise, we will make use of the mechanism in \cite{DW2014} by 
solving the \emph{strong formulation} of the rank-dependent mean field game without common noise and then by observing that purely rank-dependent reward functions are translation invariant. 

The rest of the paper is organized as follows: In Section~\ref{sec:Nplayer} we introduce the N-player game in which the players are coupled through the reward function which is rank-based. In Section~\ref{sec:IN} we consider the case without common noise. We first find the mean field limit, discuss the uniqueness of the Nash equilibrium, and construct an approximate Nash equilibrium using the mean field limit. Using these results, in Section~\ref{sec:CN}  we use the mechanism in \cite{DW2014} and obtain respective results for the common noise.

\section{The $N$-player game}\label{sec:Nplayer}

%A tournament, broadly speaking, is a compensation scheme that is based on relative performance. 
We consider $N$ players each of whom controls her own state variable and is rewarded based on her rank.
We will denote by $X_i$ the $i$-th player's state variable, and assume that it satisfies the following stochastic differential equation (SDE)
\[dX_{i,t}=a_{i,t}dt+\sigma dB_{i,t}+\sigma_0 dW_t, \quad X_{i,0}=0,\]
where $a_{i}$ is the control by agent $i$, and $(B_{i})_{i=1,\ldots, N}$ and $W$ are independent standard Brownian motions defined on some filtered probability space $(\Omega, \mathcal{F}, \{\mathcal{F}_t\}_{t\in[0,T]}, \mathbb{P})$, representing the idiosyncratic noises and common noise, respectively. The game ends at time $T>0$, when each player receives a rank-based reward minus the running cost of effort, which we will assumed to be quadratic $c a^2$ for some constant $c>0$. 

In order to precisely define the rank-based reward, let
\[\bar\mu^N:=\frac{1}{N}\sum_{i=1}^N \delta_{X_{i,T}}\] 
denote the empirical measure of the terminal state of the $N$-player system. Then $\bar\mu^N(-\infty, X_{i,T}]$ 
gives the fraction of players that finish the same or worse than player $i$. Let $\mathbb{R}\times [0,1]\ni (x,r)\mapsto R(x,r)\in \mathbb{R}$ be a bounded continuous function that is non-decreasing in both arguments. For any probability measure $\mu$ on $\mathbb{R}$, write $R_\mu(x)=R(x,\mu(-\infty,x])=R(x,F_\mu(x))$ where $F_\mu$ denotes the cumulative distribution function of $\mu$. The reward player $i$ receives is given by 
\[R_{\bar\mu^N}(X_{i,T})=R(X_{i,T},{\bar\mu^N}(-\infty,X_{i,T}])=R(X_{i,T},F_{\bar\mu^N}(X_{i,T})).\]
When $R(x,r)$ is independent of $x$, the compensation scheme is purely rank-based. In general, we could have a mixture of absolute performance compensation and relative performance compensation. The objective of each player is to observe the progress of all players and choose her effort level to maximize the expected payoff, while anticipating the other players' strategies. 

The players' equilibrium expected payoffs, as functions of time and state variables, satisfy a system of $N$ coupled nonlinear partial differential equations subject to discontinuous boundary conditions, which appears to be analytically intractable. Fortunately, in a large-population game, the impact of any individual on the whole population is very small. So it is often good enough for each player to ignore the private state of any other individual and simply optimize against the aggregate distribution of the population. As a consequence, the equilibrium strategies decentralize in the limiting game as $N\rightarrow \infty$. We shall use the mean field limit to construct approximate Nash equilibrium for the $N$-player game, both in the case with and without common noise.

\section{Mean field approximation when there is no common noise}\label{sec:IN}

In this section, we assume $\sigma_0=0$. Solving the mean field game consists of two sub-problems: a stochastic control problem and a fixed-point problem (also called the consistency condition). For any Polish space $\mathcal{X}$, denote by $\mathcal{P}(\mathcal{X})$ the space of probability measures on $\mathcal{X}$, and $\mathcal{P}_1(\mathcal{X}):=\{\mu\in \mathcal{P}(\mathcal{X}): \int_{\mathcal{X}}|x| d\mu(x)<\infty\}$. 

We first fix a distribution $\mu\in\mathcal{P}(\mathbb{R})$ of the terminal state of the population, and consider a single player's optimization problem:
\begin{equation}\label{agent-problem}
v(t,x):=\sup_{a} \mathbb{E}_{t,x}\left[R_\mu(X_T)-\int_t^T c a_s^2ds\right]
\end{equation}
where
\begin{equation}\label{agent-problem-SDE}
dX_s=a_s ds+\sigma dB_s,
\end{equation}
$B$ is a Brownian motion, and $a$ ranges over the set of progressively measurable processes satisfying $\mathbb{E}\int_0^T |a_s| ds<\infty$.
The associated dynamic programming equation is
\[v_t+\sup_{a}\left\{av_x+\frac{1}{2}\sigma^2 v_{xx}-ca^2\right\}=0\]
with terminal condition $v(T,x)=R_\mu(x)$. Using the first-order condition, we obtain that the candidate optimizer is $a^\ast=\frac{v_x}{2c}$, and the Hamilton-Jacobi-Bellman (HJB) equation can be written as
\[v_t+\frac{1}{2}\sigma^2 v_{xx}+\frac{(v_x)^2}{4c}=0.\]
The above equation can be linearized using the Cole-Hopf transformation $u(t,x):=e^{(2c\sigma^2)^{-1}v(t,x)}$, giving
\[u_t+\frac{1}{2}\sigma^2 u_{xx}=0.\]
Together with the boundary condition $u(T,x)=e^{(2c\sigma^2)^{-1} R_\mu(x)}$, we can easily write down the solution:
\begin{equation}\label{u}
u(t,x)=\mathbb{E}\left[\exp\left(\frac{1}{2c\sigma^2}R_\mu(x+\sigma \sqrt{T-t}Z)\right)\right]
\end{equation}
where $Z$ is a standard normal random variable. Let us further write $u$ as an integral:
\begin{align*}
u(t,x)&=\int_{-\infty}^\infty \exp\left(\frac{1}{2c\sigma^2}R_\mu(x+\sigma \sqrt{T-t}z)\right) \frac{1}{\sqrt{2\pi}}\exp\left(-\frac{z^2}{2}\right)dz\\
&=\int_{-\infty}^\infty \exp\left(\frac{1}{2c\sigma^2}R_\mu(y)\right) \frac{1}{\sqrt{2\pi\sigma^2(T-t)}}\exp\left(-\frac{(y-x)^2}{2\sigma^2(T-t)}\right)dy.
\end{align*}
Using the dominated convergence theorem, we can differentiate under the integral sign and get
\begin{align*}
u_x(t,x)&=\int_{-\infty}^\infty \exp\left(\frac{1}{2c\sigma^2}R_\mu(y)\right) \frac{1}{\sqrt{2\pi\sigma^2(T-t)}}\exp\left(-\frac{(y-x)^2}{2\sigma^2(T-t)}\right) \frac{(y-x)}{\sigma^2(T-t)}dy \\
&=\int_{-\infty}^\infty \exp\left(\frac{1}{2c\sigma^2}R_\mu(x+\sigma\sqrt{T-t}z)\right) \frac{1}{\sqrt{2\pi}}\exp\left(-\frac{z^2}{2}\right) \frac{z}{\sigma\sqrt{T-t}}dz\\ 
&=\mathbb{E}\left[\exp\left(\frac{1}{2c\sigma^2}R_\mu(x+\sigma \sqrt{T-t}Z)\right)\frac{Z}{\sigma\sqrt{T-t}}\right]. \numberthis \label{u_x}
\end{align*}
Similarly, we obtain
\begin{equation}\label{u_xx}
u_{xx}=\mathbb{E}\left[\exp\left(\frac{1}{2c\sigma^2} R_\mu (x+\sigma \sqrt{T-t}Z)\right)\frac{Z^2-1}{\sigma^2(T-t)}\right].
\end{equation}
Using \eqref{u}-\eqref{u_xx}, together with the boundedness and monotonicity of $R$, we easily get the following estimates. Note that all bounds are independent of $\mu$.
\begin{lemma} \label{lemma:bdd}
The functions $u$ and $v$ satisfy
\begin{align*}
0<K^{-1}\leq u(t,x)\leq K,\quad & -\snorm{R}\leq v(t,x)\leq \snorm{R},\\
0\leq u_x(t,x)\leq \frac{K}{\sigma}\sqrt{\frac{2}{\pi}}\frac{1}{\sqrt{T-t}},\quad & 0\leq v_x(t,x)\leq 2c\sigma K^2\sqrt{\frac{2}{\pi}} \frac{1}{\sqrt{T-t}},\\
|u_{xx}(t,x)|\leq \frac{2K}{\sigma^2}\frac{1}{T-t},\quad & |v_{xx}(t,x)|\leq \frac{4cK^2(1+K^2\pi^{-1})}{T-t},
\end{align*}
where $K:=\exp((2c\sigma^2)^{-1}\snorm{R})$.
\end{lemma}
Since $v_{xx}$ is bounded, the drift coefficient $a^\ast=\frac{v_x}{2c}$ is Lipschitz continuous in $x$. It follows that the optimally controlled state process, denoted by $X^\ast$, has a strong solution on $[0,T)$. Observe that
\[0\leq\int_t^T a^\ast(s,X^\ast_s)ds\leq \int_t^T \frac{\sigma K^2 \sqrt{2/\pi}}{\sqrt{T-s}}ds=2\sigma K^2 \sqrt{\frac{2(T-t)}{\pi}} <\infty.\]
So the optimal cumulative effort is bounded by some constant independent of $\mu$. It also implies that $X^\ast_u=x+\int_t^u a^\ast(s,X^\ast_s)ds+\sigma (B_u-B_t)$ has a well-defined limit as $u\rightarrow T$. Standard verification theorem yields that the solution to the HJB equation is the value function of the problem \eqref{agent-problem}-\eqref{agent-problem-SDE}, and that $a^\ast$ is the optimal Markovian feedback control. Finally, using the dominated convergence theorem again, we can show that for $t<T$,
\[\lim_{x\rightarrow \pm\infty} u_x(t,x)=0.\]
The same limits also hold for $a^\ast$ since $u$ is bounded away from zero.
In other words, the optimal effort level is small when the progress is very large in absolute value. This agrees with many real life observations that when a player has a very big lead, it is easy for her to show slackness; and when one is too far behind, she often gives up on the game instead of trying to catch up.

\subsection{Existence of a Nash equilibrium}\label{subsec:existenceNE}
For each fixed $\mu\in\mathcal{P}(\mathbb{R})$, solving the stochastic control problem \eqref{agent-problem}-\eqref{agent-problem-SDE} yields a value function $v(t,x;\mu)$ and a best response $a^\ast(t,x)=(2c)^{-1} v_x(t,x;\mu)$. Suppose the game is started at time zero, with zero initial progress, the optimally controlled state process $X^\mu$ of the generic player satisfies the SDE
\begin{equation}\label{SDE}
dX_t=\frac{v_x(t,X_t;\mu)}{2c}dt+\sigma dB_t, \ X_0=0.
\end{equation}
Finding a Nash equilibrium for the limiting game is equivalent to finding a fixed point of the mapping $\Phi: \mu\mapsto \mathcal{L}(X^\mu_{T})$, where $\mathcal{L}(\cdot)$ denotes the law of its argument. We shall sometimes refer to such a fixed point as an equilibrium measure.

\begin{theorem}
The mapping $\Phi$ has a fixed point.
\end{theorem}
\begin{proof}
Similar to \cite{CarmonaDelarue13}, we will use Schauder's fixed point theorem. Observe that for any $\mu\in\mathcal{P}(\mathbb{R})$, we have
\[\mathbb{E}\left[ |X^\mu_T|^2\right]\leq \mathbb{E}\left[\left(2\sigma K^2 \sqrt{\frac{2T}{\pi}}+\sigma |B_T|\right)^2\right]=:C_0.\]
This implies the set of $\Phi(\mu)=\mathcal{L}(X^\mu_T)$ is tight in $\mathcal{P}(\mathbb{R})$, hence relatively compact for the topology of weak convergence by Prokhorov theorem. Recall that $\mathcal{P}_1(\mathbb{R})=\{\mu\in \mathcal{P}(\mathbb{R}): \int_{\mathbb{R}}|x| d\mu(x)<\infty\}$. Equip $\mathcal{P}_1(\mathbb{R})$ with the topology induced by the 1-Wasserstein metric:
\begin{align*}
W_1(\mu, \mu')&:=\inf\left\{\int_{\mathbb{R}^2} |x-y| d\pi(x,y): \pi\in \mathcal{P}_1(\mathbb{R}^2) \text{ with marginals } \mu \text{ and } \mu'\right\}\\
&=\sup \left\{\int_{\mathbb{R}}\psi d\mu-\int_{\mathbb{R}}\psi d\mu': \psi\in \text{Lip}_1(\mathbb{R})\right\}.
\end{align*}
Here $\text{Lip}_1(\mathbb{R})$ denotes the space of Lipschitz continuous functions on $\mathbb{R}$ whose Lipschitz constant is bounded by one. It is known that $(\mathcal{P}_1(\mathbb{R}), W_1)$ is a complete separable metric space (see e.g.\ \cite[Theorem 6.18]{OptimalTransport2009}).
We shall work with a subset of $\mathcal{P}_1(\mathbb{R})$ defined by
\[\mathcal{E}:=\left\{\mu\in \mathcal{P}_1(\mathbb{R}): \int_{\mathbb{R}}|x|^2 d\mu(x)\leq C_0\right\}.\]
It is easy to check that $\mathcal{E}$ is non-empty, convex and closed (for the topology induced by the $W_1$ metric). Moreover, one can show using \cite[Definition 6.8(iii)]{OptimalTransport2009} that any weakly convergent sequence $\{\mu_n\}\subseteq \mathcal{E}$ is also $W_1$-convergent. Therefore, $\mathcal{E}$ is also relatively compact for the topology induced by the $W_1$ metric. So we have found a non-empty, convex and compact set $\mathcal{E}$ such that $\Phi$ maps $\mathcal{E}$ into itself. It remains to show $\Phi$ is continuous on $\mathcal{E}$. In the rest of the proof, the constant $C$ may change from line to line.

Let $\{\mu_k\}\subseteq\mathcal{E}$ such that $W_1(\mu_k, \mu)\rightarrow 0$ as $k\rightarrow \infty$. We wish to show $W_1(\Phi(\mu_k), \Phi(\mu))\rightarrow 0$. Note that
\[W_1(\Phi(\mu_k), \Phi(\mu))\leq \mathbb{E}\left[|X^{\mu_k}_T-X^\mu_T|\right]\leq \frac{1}{2c}\int_0^T  \mathbb{E} \left[|v_x(t,X^{\mu_k}_t;\mu_k)-v_x(t,X^{\mu}_t;\mu)|\right]dt.\]
From Lemma~\ref{lemma:bdd}, we know that
$|v_x(t,X^{\mu_k}_t;\mu_k)-v_x(t,X^{\mu}_t;\mu)|\leq \frac{C}{\sqrt{T-t}}$.
Since $\int_0^T \frac{C}{\sqrt{T-t}}dt<\infty$, thanks to the dominated convergence theorem, it suffices to show for $t\in[0,T)$,
\[  \mathbb{E}\left[ |v_x(t,X^{\mu_k}_t;\mu_k)-v_x(t,X^{\mu}_t;\mu)|\right]\rightarrow 0.\]

By Lemma~\ref{lemma:bdd} and the mean value theorem, we have that
\begin{align*}
&|v_x(t,X^{\mu_k}_t;\mu_k)-v_x(t,X^{\mu}_t;\mu)|\\
&\leq |v_x(t,X^{\mu_k}_t;\mu_k)-v_x(t,X^{\mu}_t;\mu_k)|+|v_x(t,X^{\mu}_t;\mu_k)-v_x(t,X^{\mu}_t;\mu)|\\
&\leq \frac{C}{T-t}|X^{\mu_k}_t-X^\mu_t|+|v_x(t,X^{\mu}_t;\mu_k)-v_x(t,X^{\mu}_t;\mu)|.
\end{align*}
So to show $W_1(\Phi(\mu_k), \Phi(\mu))\rightarrow 0$, it suffices to show that for each fixed $t\in[0,T)$, 
\begin{equation}\label{claim1}
\mathbb{E}\left[|v_x(t,X^{\mu}_t;\mu_k)-v_x(t,X^{\mu}_t;\mu)|\right]\rightarrow 0,
\end{equation}
and 
\begin{equation}\label{claim2}
\mathbb{E}\left[|X^{\mu_k}_t-X^\mu_t|\right]\rightarrow 0.
\end{equation}

We first show \eqref{claim1}. Using the estimates in Lemma~\ref{lemma:bdd}, we get
\begin{align*}
&\mathbb{E} \left[|v_x(t,X^{\mu}_t;\mu_k)-v_x(t,X^{\mu}_t;\mu)|\right]\\
&=C\mathbb{E}\left[\left|\frac{u(t,X^\mu_t;\mu)[u_x(t,X^\mu_t; \mu_k)-u_x(t,X^\mu_t;\mu)]+u_x(t,X^\mu_t;\mu)[u(t,X^\mu_t;\mu)-u(t,X^\mu_t;\mu_k)]}{u(t,X^\mu_t;\mu_k)u(t,X^\mu_t;\mu)}\right|\right]\\ %C=2c\sigma^2
&\leq C \mathbb{E}\left[|u_x(t,X^\mu_t; \mu_k)-u_x(t,X^\mu_t;\mu)|\right]+\frac{C}{\sqrt{T-t}}\mathbb{E}\left[|u(t,X^\mu_t;\mu)-u(t,X^\mu_t;\mu_k)|\right].
\end{align*}
Since all integrands are bounded, to show the expectations converge to zero, it suffices to check that the integrands converge to zero a.s. Fix $\omega\in\Omega$, we know from \eqref{u_x} that
\begin{align*}
&\left|u_x(t,X^\mu_t(\omega);\mu_k)-u_x(t,X^\mu_t(\omega);\mu)\right|\\
&\leq C\mathbb{E}\left[\frac{|Z|}{\sigma\sqrt{T-t}} \left|R_{\mu_k}(x+\sigma\sqrt{T-t}Z)-R_\mu(x+\sigma\sqrt{T-t}Z)\right|\right]_{x=X^\mu_t(\omega)}.
\end{align*}
Since $W_1(\mu_k, \mu)\rightarrow 0$, $\mu_k$ also converges to $\mu$ weakly, and the cumulative distribution function $F_{\mu_k}(x)$ converges to $F_\mu(x)$ at every point $x$ at which $F_\mu$ is continuous. It follows from the continuity of $R$ that $R_{\mu_k}(x)$ converges to $R_\mu(x)$ at every point $x$ at which $F_\mu$ is continuous. Since $F_\mu$ has at most countably many points of discontinuity, the random variable inside the expectation converges to zero a.s. The dominated convergence theorem then allows us to interchange the limit and the expectation, giving that
\[|u_x(t,X^\mu_t(\omega); \mu_k)-u_x(t,X^\mu_t(\omega);\mu)|\rightarrow 0.\]
Similarly, from \eqref{u} we obtain
\begin{align*}
&\left|u(t,X^\mu_t(\omega), \mu_k)-u(t,X^\mu_t(\omega),\mu)\right|\\
&\leq C\mathbb{E} \left[\left|R_{\mu_k}(x+\sigma\sqrt{T-t}Z)-R_\mu(x+\sigma\sqrt{T-t}Z)\right|\right]_{x=X^\mu_t(\omega)}.
\end{align*}
Again, using that $F_\mu$ has countably many points of discontinuity, one can show that
\[\left|u(t,X^\mu_t(\omega), \mu_k)-u(t,X^\mu_t(\omega),\mu)\right|\rightarrow 0.\]
Putting everything together, we have proved \eqref{claim1}. 

Next, we show \eqref{claim2} by Gronwall's inequality. Let $\epsilon>0$ be given. For any $r\in[0,t]$,
\begin{align*}
\mathbb{E}\left[|X^{\mu_k}_r-X^\mu_r|\right]&\leq \frac{1}{2c}\int_0^r \mathbb{E}\left[\left|v_x(s,X^{\mu_k}_s;\mu_k)-v_x(s,X^{\mu}_s;\mu)\right|\right]ds\\
&\leq \int_0^r \mathbb{E}\left[ \frac{C}{T-s}|X^{\mu_k}_s-X^\mu_s|+\frac{1}{2c}|v_x(s,X^{\mu}_s;\mu_k)-v_x(s,X^{\mu}_s;\mu)|\right]ds.
\end{align*}
By \eqref{claim1} and the bounded convergence theorem, we obtain
\[\int_0^t \mathbb{E}\left[|v_x(s,X^{\mu}_s;\mu_k)-v_x(s,X^{\mu}_s;\mu)|\right] ds\rightarrow 0.\]
So for $k$ large enough, we have
\begin{align*}
\mathbb{E}\left[|X^{\mu_k}_r-X^\mu_r|\right]
&\leq \frac{C}{T-t}\int_0^r \mathbb{E}\left[|X^{\mu_k}_s-X^\mu_s|\right]ds+\epsilon e^{-\frac{Ct}{T-t}}.
\end{align*}
By Gronwall's inequality, 
\[\mathbb{E}\left[|X^{\mu_k}_t-X^\mu_t|\right]\leq \epsilon e^{-\frac{Ct}{T-t}}+\frac{C}{T-t}\int_0^t \epsilon e^{-\frac{Ct}{T-t}}e^{\frac{C(t-s)}{T-t}}ds=\epsilon.\]
This completes the proof of \eqref{claim2}, and thus the continuity of $\Phi$. By Schauder's fixed point theorem, there exists a fixed point of $\Phi$ in the set $\mathcal{E}$.
\end{proof}

\subsection{Uniqueness of Nash equilibrium.} \label{sec:uniqueness}
Let $\mathcal{C}\subseteq\mathcal{P}(\mathbb{R})$ be a class of measures in which uniqueness will be established.
We first state a monotonicity assumption which is in the spirit of \cite{MR2295621}. 
\begin{assumption}\label{assumption-uniqueness}
For any $\mu, \mu'\in\mathcal{C}$, we have
\[\int_\mathbb{R} (R_\mu-R_{\mu'})(x) d(\mu-\mu')(x)\leq 0.\]
\end{assumption}
\begin{remark}
Take $\mathcal{C}$ to be the set of all measures in $\mathcal{P}(\mathbb{R})$ that are absolutely continuous with respect to the Lebesgue measure, then Assumption~\ref{assumption-uniqueness} is satisfied if the reward function $R$ is Lipschitz continuous and
\begin{equation*}
h(x,r_1,r_2):=\frac{R(x,r_1)-R(x,r_2)}{r_1-r_2} , \quad x\in \mathbb{R}, (r_1,r_2)\in [0,1]^2\backslash \{r_1=r_2\}
\end{equation*}
is differentiable and has non-negative partial derivatives $h_x, h_{r_1}, h_{r_2}$. This includes any continuously differentiable function $R$ which satisfies (i) $r\mapsto R(x,r)$ is convex, and (ii) $r\mapsto R_x(x,r)$ is non-decreasing. %In particular, all purely rank-dependent convex reward schemes fall into this class.
%An example of such a reward function is
%\[R(x,r)=A_1(x)\mathfrak{R}(r)+A_2(x)\]
%where $A_1, A'_1, A'_2\geq 0$ and $\mathfrak{R}$ is any continuously differentiable increasing convex function.
To see why $h_x, h_{r_1}, h_{r_2}\geq 0$ is sufficient to verify Assumption~\ref{assumption-uniqueness}, first note that for any $\mu, \mu'\in\mathcal{C}$, $R_\mu$ and $R_{\mu'}$ are absolutely continuous. Using integration by parts for absolutely continuous functions, we have
 \begin{align*}
&\int_\mathbb{R} (R_{\mu}-R_{\mu'})(x) d(\mu-\mu')(x)=\int_\mathbb{R} (F_{\mu}-F_{\mu'})(x)h(x,F_{\mu}(x),F_{\mu'}(x)) d(F_{\mu}-F_{\mu'})(x)\\
 &=-\int_\mathbb{R}(F_{\mu}-F_{\mu'})(x) d\left[(F_{\mu}-F_{\mu'})(x)h(x,F_{\mu}(x),F_{\mu'}(x))\right]\\
 %&= -\int_\mathbb{R}(F_{\mu}-F_{\mu'})(x) \left[(F_{\mu}-F_{\mu'})(x)dh(x,F_{\mu}(x),F_{\mu'}(x))+h(x,F_{\mu}(x),F_{\mu'}(x))d(F_{\mu}-F_{\mu'})(x)\right]\\
 &= -\int_\mathbb{R}(F_{\mu}-F_{\mu'})^2(x) dh(x,F_{\mu}(x),F_{\mu'}(x)) -\int_\mathbb{R} (R_{\mu}-R_{\mu'})(x) d(\mu-\mu')(x)
 \end{align*}
Re-arranging terms and using that $h_x, h_{r_1}, h_{r_2}\geq 0$, we get
 \begin{align*}
 &2\int_\mathbb{R} (R_{\mu}-R_{\mu'})(x) d(\mu-\mu')(x)\\
 &= -\int_\mathbb{R}(F_{\mu}-F_{\mu'})^2(x) \nabla h (x,F_\mu(x),F_{\mu'}(x)) \cdot (dx, dF_\mu(x),dF_{\mu'}(x))\leq 0.
 \end{align*}
 
 If one measures the rank of $x$ with respect to a given distribution $\mu$ using the "regular" cumulative distribution function $\tilde{F}_\mu(x):=\frac{1}{2}(F_\mu(x+)+F_\mu(x-))$, then for the case $R(x,r)=r$, Assumption~\ref{assumption-uniqueness} is satisfied with $\mathcal{C}=\mathcal{P}(\mathbb{R})$ (see \cite[Theorem B]{Hewitt1960}).
\end{remark}

\begin{proposition}
Under Assumption~\ref{assumption-uniqueness}, $\Phi$ has at most one fixed point in $\mathcal{C}$.
\end{proposition}
\begin{proof}
Suppose $\mu$ and $\mu'$ are two fixed points of $\Phi$ in $\mathcal{C}$. To simplify notation, write $v(t,x):=v(t,x;\mu)$ and $v'(t,x):=v(t,x;\mu')$. 
Let $X^\mu$ and $X^{\mu'}$ be the optimally controlled state processes (starting at zero) in response to $\mu$ and $\mu'$, respectively. Let $t\in(0,T)$. Using It\^o's lemma and the PDE satisfied by $v$ and $v'$, it is easy to show that
\begin{equation}\label{uniqueness-eq-1}
\mathbb{E} \left[v(t,X^\mu_t)\right]=v(0,0)+\mathbb{E}\left[\int_0^t \frac{1}{4c}(v_x)^2(s,X^\mu_s)ds\right],
\end{equation}
and 
\begin{equation}\label{uniqueness-eq-2}
\mathbb{E} \left[v'(t,X^\mu_t)\right]=v'(0,0)+\mathbb{E}\left[\int_0^t \frac{1}{4c}\left[2v'_xv_x -(v'_x)^2\right](s,X^\mu_s) ds\right].
\end{equation}
Write $\Delta v:=v-v'$, we obtain by subtracting \eqref{uniqueness-eq-2} from \eqref{uniqueness-eq-1} that
\[\mathbb{E}\left[\Delta v(t,X^\mu_t)\right]=\Delta v(0,0)+\mathbb{E}\left[\int_0^t \frac{1}{4c}[(\Delta v)_x(s,X^\mu_s)]^2ds\right].\]
Letting $t\rightarrow T$ and using the continuity of $v$ and $v'$ at the terminal time, we get
\begin{equation}\label{uniqueness-eq-3}
\mathbb{E}\left[(R_\mu-R_{\mu'})(X^\mu_T)\right]=\mathbb{E}\left[\Delta v(T,X^\mu_T)\right]=\Delta v(0,0)+\mathbb{E}\left[\int_0^T \frac{1}{4c}[(\Delta v)_x(s,X^\mu_s)]^2ds\right].
\end{equation}
Now, exchange the role of $\mu$ and $\mu'$. We also have
\begin{equation}\label{uniqueness-eq-4}
\mathbb{E}[(R_{\mu'}-R_{\mu})(X^{\mu'}_T)]=-\mathbb{E}[\Delta v(T,X^{\mu'}_T)]=-\Delta v(0,0)+\mathbb{E}\left[\int_0^T \frac{1}{4c}[(\Delta v)_x(s,X^{\mu'}_s)]^2ds\right].
\end{equation}
Adding \eqref{uniqueness-eq-3} and \eqref{uniqueness-eq-4}, and using that $\mu=\mathcal{L}(X^\mu_T), \mu'=\mathcal{L}(X^{\mu'}_T)$, we get
\begin{align*}
0&\leq \frac{1}{4c}\mathbb{E}\left[\int_0^T [(\Delta v)_x(s,X^\mu_s)]^2+[(\Delta v)_x(s,X^{\mu'}_s)]^2ds\right]\\
&=\mathbb{E}\left[(R_\mu-R_{\mu'})(X^\mu_T)\right]+\mathbb{E}[(R_{\mu'}-R_{\mu})(X^{\mu'}_T)]=\int_\mathbb{R} (R_\mu-R_{\mu'})(x) d(\mu-\mu')(x)\leq 0,
\end{align*}
where the last inequality follows from Assumption~\ref{assumption-uniqueness}. This implies
\[v_x(s,X^{\mu'}_s)=v'_x(s,X^{\mu'}_s) ~d\mathbb{P}\times dt\text{-a.e.}\]
By the uniqueness of the solution of the SDE \eqref{SDE}, we must have $X^\mu_T=X^{\mu'}_T$ a.s.\ and $\mu=\mu'$.
\end{proof}

\subsection{Approximate Nash equilibrium of the $N$-player game} \label{subsec:ApproxNE}
The MFG solution allows us to construct, using decentralized strategies, an approximate Nash equilibrium of the $N$-player game when $N$ is large. In the MFG literature, this is typically done using results from the propagation of chaos. Here we have a simpler problem since the mean-field interaction does not enter the dynamics of the state process. And it is this special structure that allows us to handle rank-based terminal payoff which fails to be Lipschitz continuous in general.

\begin{definition}
A progressively measurable vector $\mathbf{a}=(a_{1}, \ldots, a_N)$ is called an $\epsilon$-Nash equilibrium of the $N$-player game if 
\begin{itemize}
\item[(i)] $\mathbb{E}\left[\int_0^T |a_{i,t}| dt\right]<\infty$ for any $i\in\{1, \ldots, N\}$; and
\item[(ii)] for any $i\in\{1, \ldots, N\}$, and any progressively measurable process $\beta$ satisfying $\mathbb{E}\left[\int_0^T |\beta_t| dt\right]<\infty$, we have
\[\mathbb{E}\left[R_{\bar{\mu}^{N,\mathbf{a}}}(X^{a_i}_{i,T})-\int_0^T c a_{i,t}^2dt\right]+\epsilon\ge \mathbb{E}\left[R_{\bar{\mu}^{N,\mathbf{a}^i_\beta}}(X^\beta_{i,T})-\int_0^T c \beta_t^2dt\right],\]
where $X^\beta_{i,T}=\int_0^T \beta_t dt+\sigma B_{i,T}$, $\mathbf{a}^i_\beta=(a_1, \ldots, a_{i-1}, \beta, a_{i+1}, \ldots, a_N)$ and $\bar{\mu}^{N,\mathbf{a}}=\frac{1}{N}\sum_{j=1}^N \delta_{X^{a_j}_{j,T}}$.
\end{itemize}
\end{definition}

We now state an additional H\"older condition on $R$ which allows us to get the convergence rate. It holds, for example, when $R(x,r)=A(x)r^p+B(x)$ where $p\in(0,\infty)$ and $A\in L^\infty(\mathbb{R})$.
\begin{assumption}\label{assumption:R}
There exist constants $L>0$ and $\alpha\in(0,1]$ such that $|R(x,r_1)-R(x,r_2)|\leq L|r_1-r_2|^\alpha$ for any $r_1, r_2\in[0,1]$ and $x\in\mathbb{R}$.
\end{assumption}
\begin{theorem}\label{thm:epsNE}
Let Assumption~\ref{assumption:R} hold. For any fixed point $\mu$ of $\Phi$, 
\[\bar{a}_{i,t}:=(2c)^{-1}v_x(t,X^{\bar{a}_i}_{i,t};\mu), \ i=1,\ldots, N\] form an $O(N^{-\alpha/2})$-Nash equilibrium of the $N$-player game as $N\rightarrow \infty$.
\end{theorem}
\begin{proof}
Let $\mu$ be a fixed point of $\Phi$, and let $\bar{a}_{i,t}$ be defined as in the theorem statement. To keep the notation simple, we omit the superscript of any state process if it is controlled by the optimal Markovian feedback strategy $(2c)^{-1}v_x(t,x;\mu)$. Let 
\[V:=v(0,0;\mu)=\mathbb{E}\left[R_\mu(X_{T})-\int_0^T \frac{1}{4c}v_x^2(s,X_s;\mu)ds\right]\]
be the value of the limiting game where $X$ satisfies \eqref{SDE}, and
\[J^{N}_i:=\mathbb{E}\left[R_{\bar{\mu}^N}(X_{i,T})-\int_0^T c\bar{a}^2_{i,s}ds\right]\]
be the net gain of player $i$ in an $N$-player game, if everybody use the candidate approximate Nash equilibrium $(\bar{a}_1,\ldots, \bar{a}_N)$. Here $\bar{\mu}^N=\frac{1}{N}\sum_{i=1}^N\delta_{X_{i,T}}$. Since our state processes do not depend on the empirical measure (the interaction is only through the terminal payoff), each $X_i$ is simply an independent, identical copy of $X$. Hence
\[V=\mathbb{E}\left[R_\mu(X_{i,T})-\int_0^T c\bar{a}^2_{i,s}ds\right].\]
Let us first show that $J^N_i$ and $V$ are close. We have
\[J^N_i-V=\mathbb{E}[R_{\bar{\mu}^N}(X_{i,T})-R_{\mu}(X_{i,T})].\]
It follows from the $\alpha$-H\"older continuity of $R$ that
\begin{align*}
|J^N_i-V|&
%\leq \mathbb{E}[|R_{\bar{\mu}^N}(X_{i,T})-R_{\mu}(X_{i,T})|]
\leq L \mathbb{E}[|F_{\bar{\mu}^N}(X_{i,T})-F_\mu(X_{i,T})|^\alpha]\leq L \mathbb{E}[\|\hat{F}^{N}_\mu-F_\mu\|^\alpha_\infty],
\end{align*}
where for $n\in\mathbb{N}$, $\hat{F}^{n}_\mu$ denotes the empirical cumulative distribution function of $n$ i.i.d.\ random variables with cumulative distribution function $F_\mu$. By Dvoretzky-Kiefer-Wolfowitz inequality, we have
\[\mathbb{P}\left(\|\hat{F}^{N}_\mu-F_\mu\|_\infty>\epsilon\right)\leq 2e^{-2N\epsilon^2}.\]
It follows that
\begin{align*}
|J^N_i-V|&\leq L\mathbb{E}[\|\hat{F}^{N}_\mu-F_\mu\|^\alpha_\infty]=L\int_0^\infty\mathbb{P}\left(\|\hat{F}^{N}_\mu-F_\mu\|^\alpha_\infty>z\right) dz\\
&\leq L\int_0^\infty2e^{-2Nz^{2/\alpha}} dz=\frac{2L}{(4N)^{\alpha/2}}\int_0^\infty e^{-\frac{1}{2}y^{2/\alpha}} dy\\
&=O(N^{-\alpha/2})~\text{ as } N\rightarrow \infty.
\end{align*}

Next, consider the system where player $i$ makes a unilateral deviation from the candidate approximate Nash equilibrium $(\bar{a}_1,\ldots, \bar{a}_N)$; say, she chooses an admissible control $\beta$. Denote her controlled state process by $X^\beta_i$, and the state processes of all other players by $X_j$ as before for $j\neq i$. Let $\bar{\nu}^N:=\frac{1}{N}(\delta_{X^\beta_{i,T}}+\sum_{j\neq i} \delta_{X_{j,T}})$ be the corresponding empirical measure of the terminal states, and
\[J^{N,\beta}_i:=\mathbb{E}\left[R_{\bar{\nu}^N}(X^\beta_{i,T})-\int_0^T c\beta^2_{s}ds\right]\]
be the corresponding net gain for player $i$. We have
\begin{align*}
J^{N,\beta}_i-V&=\mathbb{E}\left[R_{\bar{\nu}^N}(X^\beta_{i,T})-\int_0^T c\beta^2_{s}ds\right]-\mathbb{E}\left[R_\mu(X_{i,T})-\int_0^T c\bar{a}^2_{i,s}ds\right]\\
&=\mathbb{E}\left[R_{\bar{\nu}^N}(X^\beta_{i,T})-R_\mu(X^\beta_{i,T})\right]\\
&\quad +\mathbb{E}\left[R_\mu(X^\beta_{i,T})-\int_0^T c\beta^2_{s}ds\right] -\mathbb{E}\left[R_\mu(X_{i,T})-\int_0^T c\bar{a}^2_{i,s}ds\right]\\
&\leq \mathbb{E}\left[R_{\bar{\nu}^N}(X^\beta_{i,T})-R_\mu(X^\beta_{i,T})\right]
\end{align*}
where the inequality follows from the optimality of $\bar{a}_i$ for the $i$-th player's problem. Similar to how we estimate $|J^N_i-V|$, we have
\begin{align*}
J^{N,\beta}_i-V&\leq L\mathbb{E}[|F_{\bar{\nu}^N}(X^\beta_{i,T})-F_\mu(X^\beta_{i,T})|^\alpha]\\
&=L \mathbb{E}\left[\left|\frac{1}{N}\left(1-F_{\mu}(X^\beta_{i,T})\right)+\frac{N-1}{N}\left(\hat{F}^{N-1}_\mu(X^\beta_{i,T})-F_\mu(X^\beta_{i,T})\right)\right|^\alpha\right] \\
&\leq L\mathbb{E}\left[\left(\frac{1}{N}+\frac{N-1}{N}\|\hat{F}^{N-1}_\mu-F_\mu\|_\infty\right)^\alpha\right]\\
&\leq L\left(\frac{1}{N}+\frac{N-1}{N}\mathbb{E}\left[\|\hat{F}^{N-1}_\mu-F_\mu\|_\infty\right]\right)^\alpha
=O(N^{-\alpha/2})~\text{ as } N\rightarrow \infty,
\end{align*}
where we used Jensen's inequality in the fourth step. Combining the two estimates, we obtain
\[J^{N,\beta}_i-J^N_i\leq J^{N,\beta}_i-V+|V-J^N_i|= O(N^{-\alpha/2})~\text{ as } N\rightarrow \infty.\]
This shows $(\bar{a}_1,\ldots, \bar{a}_N)$ is an $O(N^{-\alpha/2})$-approximate Nash equilibrium.
\end{proof}
\begin{remark}
Without Assumption~\ref{assumption:R}, one can still use the continuity and boundedness of $R$ to get convergence; that is, the MFG solution still provides an approximate Nash equilibrium of the $N$-player game. However, the convergence rate is no longer valid.
\end{remark}

\section{Mean field approximation when there is common noise}\label{sec:CN}
In this section, we assume $\sigma_0>0$ and $R(x,r)$ is independent of $x$; the latter means the reward is purely rank-dependent. Unlike the case with only idiosyncratic noises, since the common noise does not average out as $N\rightarrow \infty$, the limiting environment becomes a random measure instead of a deterministic one. So the MFG problem now reads:
\begin{itemize}
\item[(i)] Fix a random measure $\mu$ of the terminal distribution of the population where the randomness comes from the common noise $W$, and solve the stochastic control problem faced by a representative player:
\begin{equation}\label{MFGCN}
V(\mu)=\sup_a \mathbb{E}\left[R_\mu(X_T)-\int_0^T \frac{1}{2}a^2_s ds\right],
\end{equation}
where
\begin{equation}\label{MFGCN-SDE}
dX_s=a_s ds+\sigma dB_s+\sigma_0 dW_s, \quad X_0=0.
\end{equation}
Denote the optimally controlled state process by $X^\mu$.
\item[(ii)] Find a fixed point of the mapping $\Psi: \mu\mapsto \mathcal{L}(X^\mu_T|W)$.
\end{itemize}

For $\mu\in\mathcal{P}(\mathbb{R})$, denote by $\mu(\cdot +q)$ the probably measure obtained by shifting $\mu$ to the left by $q\in\mathbb{R}$. Observe that when $R$ is independent of $x$, we have $R_\mu(x+q)=R(F_\mu(x+q))=R(F_{\mu(\cdot+q)}(x))=R_{\mu(\cdot+q)}(x)$. So we are precisely in the framework of translation invariant MFGs. In fact, purely rank-based functions should be another important example of translation invariant functions besides the convolution and local interaction given in \cite{DW2014}. In the general case without translation invariance, results have only been obtained in the weak formulation, see \cite{LackerAAP}.

In the remaining discussion, let us refer to the problems  \eqref{agent-problem}-\eqref{agent-problem-SDE} and \eqref{MFGCN}-\eqref{MFGCN-SDE} together with their respective fixed point problems as MFG$_{0}$ and MFG$_{cn}$, respectively. A direct application of \cite[Theorem 2.5]{DW2014} yields the following existence result.
\begin{proposition}\label{prop:REM}
Let $\bar{\mu}$ be a (deterministic) equilibrium measure of MFG$_0$. Then
\[\mu:=\bar{\mu}(\cdot -\sigma_0 W_T)\]
is a (random) equilibrium measure of MFG$_{cn}$. Moreover, the optimal control associated with $\bar{\mu}$ for MFG$_{0}$ is also an optimal open loop control associated with $\mu$ for MFG$_{cn}$.
\end{proposition}
The intuition is that the whole population is affected in parallel by the common noise. Thus, the effect of common noise is essentially cancelled out in the optimization problem due to translation invariance. Such a random equilibrium measure is clearly $\sigma(W)$-measurable, hence is a strong MFG solution in the language of \cite{CDL15}. %In fact, it only depends on $W_T$ thanks to the simple dynamics of our state process. 
\begin{remark}\label{rmk:openloop}
In Proposition~\ref{prop:REM}, the equilibrium control for both MFG$_{0}$ and MFG$_{cn}$ is $a^\ast(t,X^\circ_t;\bar{\mu})$ where $a^\ast(t,x)=(2c)^{-1}v_x(t,x;\bar{\mu})$ and $X^\circ$ is the solution to \eqref{agent-problem-SDE} controlled by $a^\ast$. Such a control is a feedback control for MFG$_{0}$, but only an open loop control for MFG$_{cn}$. Here it is reasonable to use open loop controls for MFG$_{cn}$ because for an $N$-player game, the individual can observe all state processes. When $N$ is large, the individual noises average out. Thus, observing the entire system should give each player some information about the common noise. Passing to the MFG limit, the individual should be allowed more information than that generated by her own state process.
\end{remark}

%\begin{remark}\label{rmk:openloop}
%Here we are not taking the optimal feedback form $a^\ast(t,x)=(2c)^{-1}v_x(t,x;\bar{\mu})$ and plugging in the state process $X$ with common noise, but directly using $a^\ast(t,X^\circ_t;\bar{\mu})$ where $X^\circ$ is the solution to \eqref{agent-problem-SDE} as an open loop control for MFG$_{cn}$. Such a control is independent of the common noise. Here it is reasonable to use open loop controls because for an $N$-player game, the individual can observe all state processes. When $N$ is large, the individual noises average out. Thus, observing the entire system should give some information about the common noise. Passing to the MFG limit, the agent should be allowed more information than the one generated by her own state process.
%\end{remark}

\begin{theorem}
Let Assumption~\ref{assumption:R} hold. Let $\bar{\mu}$ be an equilibrium measure of MFG$_0$,
and let $\bar{a}_{i,t}=(2c)^{-1}v_x(t,X^\circ_{i,t};\bar{\mu})$ where $X^\circ_i$ is the solution to \eqref{agent-problem-SDE} with $a$ replaced by $\bar{a}_i$ and $B$ replaced by $B_i$. Then $(\bar{a}_1,\ldots, \bar{a}_N)$ form an $O(N^{-\alpha/2})$-Nash equilibrium of the $N$-player game with common noise.
\end{theorem}
\begin{proof}
Let $\mu$ and $X^\circ$ be defined as in Proposition~\ref{prop:REM} and Remark~\ref{rmk:openloop}. Also let $X:=X^\circ+\sigma_0W$ and $X_{i}:=X^\circ_{i}+\sigma_0 W$. By Proposition~\ref{prop:REM}, we have
\[V:=V(\mu)=\mathbb{E}\left[R_\mu(X_{T})-\int_0^T \frac{1}{4c}v_x^2(s,X^\circ_s;\bar{\mu})ds\right].\]
Also let
\[J^{N}_i:=\mathbb{E}\left[R_{\bar{\mu}^N}(X_{i,T})-\int_0^T c\bar{a}^2_{i,s}ds\right]\]
be the net gain of player $i$ in an $N$-player game, if everybody use the candidate approximate Nash equilibrium $(\bar{a}_1,\ldots, \bar{a}_N)$. Here $\bar{\mu}^N:=\frac{1}{N}\sum_{i=1}^N\delta_{X_{i,T}}$ denotes the empirical measure of the terminal state of the system. Translation invariance and the definition of $\mu$ imply
\[R_{\mu}(X_T)=R_\mu(X^\circ_T+\sigma_0 W_T)=R_{\mu(\cdot+ \sigma_0 W_T)}(X^\circ_T)=R_{\bar{\mu}}(X^\circ_T).\]
Similarly, since
\[\bar{\mu}^N=\frac{1}{N}\sum_{i=1}^N \delta_{X^\circ_{i,T}+\sigma_0 W_T}=\left(\frac{1}{N}\sum_{i=1}^N \delta_{X^\circ_{i,T}}\right)(\cdot-\sigma_0 W_T)=:\bar{\mu}_\circ^N(\cdot -\sigma_0 W_T),\]
we also have
\[R_{\bar{\mu}^N}(X_{i,T})=R_{\bar{\mu}^N}(X^\circ_{i,T}+\sigma_0 W_T)=R_{\bar{\mu}^N_\circ}(X^\circ_{i,T}).\]
Hence we can rewrite $V$ and $J^N_i$ as
\[V=\mathbb{E}\left[R_{\bar{\mu}}(X^\circ_{i,T})-\int_0^T c\bar{a}_{i,s}^2ds\right] \quad\text{and}\quad J^{N}_i=\mathbb{E}\left[R_{\bar{\mu}_\circ^N}(X^\circ_{i,T})-\int_0^T c\bar{a}^2_{i,s}ds\right],\]
where we also used that $X^\circ_i$ has the same distribution as $X^\circ$. From the proof of Theorem~\ref{thm:epsNE}, we know
\[|J^N_i-V|\leq \frac{2L}{(4N)^{\alpha/2}}\int_0^\infty e^{-\frac{1}{2}y^{2/\alpha}} dy=O(N^{-\alpha/2})~\text{ as } N\rightarrow \infty.\]

Next, suppose player $i$ makes a unilateral deviation to some admissible control $\beta$. Denote her controlled state process with and without common noise by $X^\beta_i$ and $X^{\circ,\beta}_i$, respectively. We have $X^\beta_i=X^{\circ,\beta}_i+\sigma_0 W$. Let $\bar{\nu}^N:=\frac{1}{N}(\delta_{X^\beta_{i,T}}+\sum_{j\neq i} \delta_{X_{j,T}})$ and
\[J^{N,\beta}_i:=\mathbb{E}\left[R_{\bar{\nu}^N}(X^\beta_{i,T})-\int_0^T c\beta^2_{s}ds\right].\]
We have by the definition of $V$ and the H\"{o}lder continuity of $R$ that 
\begin{align*}
J^{N,\beta}_i-V&=\mathbb{E}\left[R_{\bar{\nu}^N}(X^\beta_{i,T})-\int_0^T c\beta^2_{s}ds\right]-V\\
&=\mathbb{E}\left[R_{\bar{\nu}^N}(X^\beta_{i,T})-R_\mu(X^\beta_{i,T})\right]+\mathbb{E}\left[R_\mu(X^\beta_{i,T})-\int_0^T c\beta^2_{s}ds\right]-V\\
&\leq \mathbb{E}\left[R_{\bar{\nu}^N}(X^\beta_{i,T})-R_\mu(X^\beta_{i,T})\right]\leq L\mathbb{E}\left[|\bar{\nu}^N(-\infty,X^\beta_{i,T}]-\mu(-\infty,X^\beta_{i,T}]|^\alpha\right].
\end{align*}
By translation invariance,
$\bar{\nu}^N(-\infty,X^\beta_{i,T}]=\frac{1}{N}+\left(\frac{1}{N}\sum_{j\neq i} \delta_{X^\circ_{j,T}}\right)(-\infty,X^{\circ,\beta}_{i,T}].$
By definition of $\mu$, $\mu(-\infty,X^\beta_{i,T}]=\bar{\mu}(-\infty,X^{\circ,\beta}_{i,T}]$. Also note that for $j\neq i$, $\mathcal{L}(X^\circ_{j,T})=\mathcal{L}(X^\circ_T)=\bar{\mu}$ since $\bar{\mu}$ is an equilibrium measure for MFG$_0$. Therefore, by expressing everything in terms of $\bar{\mu}$, $X^\circ$ and $X^{\circ,\beta}$, we are again back in the framework without common noise. The proof of Theorem~\ref{thm:epsNE} implies that
\begin{align*}
J^{N,\beta}_i-V&\leq L\mathbb{E}\left[\left|\frac{1}{N}+\left[\left(\frac{1}{N}\sum_{j\neq i} \delta_{X^\circ_{j,T}}\right)-\bar{\mu}\right](-\infty,X^{\circ,\beta}_{i,T}]\right|^\alpha\right]
\\&=L\mathbb{E}\left[\left|\frac{1}{N}(1-F_{\bar{\mu}}(X^{\circ,\beta}_{i,T}))+ \frac{N-1}{N} \left(\hat{F}^{N-1}_{\bar{\mu}}(X^{\circ,\beta}_{i,T})-F_{\bar{\mu}}(X^{\circ,\beta}_{i,T})\right)\right|^\alpha\right]\\
&\leq L\left(\frac{1}{N}+\frac{N-1}{N}\mathbb{E}\left[\|\hat{F}^{N-1}_{\bar{\mu}}-F_{\bar{\mu}}\|_\infty\right]\right)^\alpha
=O(N^{-\alpha/2})~\text{ as } N\rightarrow \infty,
\end{align*}
where $\hat{F}^{n}_{\bar{\mu}}$ denotes the empirical cumulative distribution function of $n$ i.i.d.\ random variables with cumulative distribution function $F_{\bar{\mu}}$. 
We conclude that $J^{N,\beta}_i-J^N_i\leq O(N^{-\alpha/2})$ and that $(\bar{a}_1,\ldots, \bar{a}_N)$ is an $O(N^{-\alpha/2})$-Nash equilibrium of the $N$-player game with common noise.
\end{proof}

\begin{remark}
The arbitrary control $\beta$ in the above proof may depend on the common noise. However, the additional information of the common noise gives each player very little advantage when everyone else use their respective $\bar{a}_{i}$'s which are independent of the common noise.
\end{remark}

\bibliographystyle{amsplain}
\bibliography{tournament}

\end{document}